\documentclass[12pt, a4paper]{amsart}

\usepackage{amsmath, amssymb, amsthm}
\usepackage{mathtools}
\usepackage[alphabetic]{amsrefs}
\usepackage{xcolor}
\usepackage{url}
\usepackage[colorlinks=true]{hyperref}

\usepackage[margin=3cm]{geometry}
\newtheorem{lemma}{Lemma}[section]
\newtheorem{theorem}[lemma]{Theorem}
\newtheorem{corollary}[lemma]{Corollary}
\newtheorem{proposition}[lemma]{Proposition}
\theoremstyle{definition}
\newtheorem{definition}[lemma]{Definition}

\newtheorem{question}[lemma]{Question}
\newtheorem{problem}[lemma]{Problem}

\newtheorem{example*}[lemma]{Example}
\newtheorem{remark}[lemma]{Remark}

\theoremstyle{remark}

\makeatletter
\newtheorem*{rep@theorem}{\rep@title}
\newcommand{\newreptheorem}[2]{%
\newenvironment{rep#1}[1]{%
\def\rep@title{{\bf #2 \ref{##1}}}%
\begin{rep@theorem}}%
{\end{rep@theorem}}}
\makeatother
\newreptheorem{theorem}{Theorem}

\DeclareRobustCommand{\qedify}[1]{%
  \ifmmode \quad\hbox{#1}
  \else
    \leavevmode\unskip\penalty9999 \hbox{}\nobreak\hfill
    \quad\hbox{#1}%
  \fi
}
\newenvironment{example}{\begin{example*}\pushQED{\qedify{$\diamondsuit$}}}{\popQED\end{example*}}

\numberwithin{equation}{section}

\newcommand{\Rbar}{\ensuremath{\overline{\mathbb R}}}

\newcommand{\I}{\mathcal{I}}

\newcommand{\B}{\mathcal{B}}
\newcommand{\F}{\mathcal{F}}

\newcommand{\ZZ}{\mathbb{Z}}

\newcommand{\NN}{\mathbb{N}}
\newcommand{\RR}{\mathbb{R}}
\newcommand{\BB}{\mathbb{B}}

\newcommand{\bfx}{\mathbf{x}}

\newcommand{\bfu}{\mathbf{u}}
\newcommand{\bfv}{\mathbf{v}}
\newcommand{\bfw}{\mathbf{w}}
\newcommand{\bfe}{\mathbf{e}}

\newcommand{\rk}{\operatorname{rk}}

\newcommand{\one}{\mathbf{1}}
\newcommand{\sat}{\mathrm{sat}}

\newcommand{\comment}[1]{}

\DeclareMathOperator{\inn}{in}

\DeclareMathOperator{\rank}{rank}

\DeclareMathOperator{\trop}{trop}

\DeclareMathOperator{\spann}{span}

\DeclareMathOperator{\cone}{cone}

\DeclareMathOperator{\starr}{star}
\DeclareMathOperator{\mon}{Mon}
\DeclareMathOperator{\mond}{{\mon_d}}
\DeclareMathOperator{\monld}{{\mon_{\leq d}}}
\DeclareMathOperator{\GL}{GL}
\DeclareMathOperator{\st}{st}

\makeatletter
\newcommand{\superimpose}[2]{{\ooalign{$#1\@firstoftwo#2$\cr\hfil$#1\@secondoftwo#2$\hfil\cr}}}
\makeatother
\newcommand{\ttimes}{\hspace{0.3mm}{\mathpalette\superimpose{{\circ}{\cdot}}}\hspace{0.3mm}}
\newcommand{\tplus}{\oplus}

\begin{document}

\title{Tropical ideals do not realise all Bergman fans}

\author{Jan Draisma}
\address{Universit\"at Bern, Mathematisches Institut,
Sidlerstrasse 5,
3012 Bern, Switzerland; and Eindhoven University of Technology}
\email{jan.draisma@math.unibe.ch}
\thanks{JD was partially supported by the Vici grant {\em Stabilisation in
Algebra and Geometry} from the Netherlands Organisation for Scientific
Research (NWO) and project grant 200021\_191981 from the
Swiss National Science Foundation (SNF)
}

\author{Felipe Rinc\'on}
\address{School of Mathematical Sciences, Queen Mary University of
London, Mile End Road, London E1 4NS, United Kingdom}
\email{f.rincon@qmul.ac.uk}
\thanks{FR was partially supported by the Research Council of Norway grant 239968/F20.}

\begin{abstract}
Every tropical ideal in the sense of Maclagan-Rinc\'on has an associated
tropical variety, a finite polyhedral complex equipped with positive
integral weights on its maximal cells. This leads to the realisability
question, ubiquitous in tropical geometry, of which weighted polyhedral
complexes arise in this manner. Using work of Las Vergnas on the
non-existence of tensor products of matroids, we prove that there is no
tropical ideal whose variety is the Bergman fan of the direct sum of the
V\'amos matroid and the uniform matroid of rank two on three elements,
and in which all maximal cones have weight one.
\end{abstract}

\maketitle

\section{Introduction}

An ideal in a polynomial ring over a field with a non-Archimedean
valuation gives rise to a tropical variety, either by taking all weight
vectors whose initial ideals do not contain a monomial or, equivalently
if the field and the value group are large enough \cite{Draisma06a}*{Theorem 4.2}, by applying the coordinate-wise valuation to all points
in the zero set of the ideal.  In the middle of this construction sits a
{\em tropical ideal}, obtained by applying the valuation to all polynomials
in the ideal. This ideal is a purely tropical object, in that
it does not know about the field or the valuation, and 
it contains more information than the tropical variety itself. For these
reasons, tropical ideals, axiomatised in \cite{MR}, were proposed as
the correct algebraic structures on which to build a theory of
tropical schemes. We review the relevant definitions below.

It was proved in \cite{MR} that tropical ideals, while not finitely generated
as ideals---nor in any sense that we know of!---have a rational
Hilbert series, satisfy the ascending chain condition, and define a
tropical variety: a finite weighted polyhedral complex. 
Later in \cite{MR2}, it was shown that the top-dimensional
parts of these varieties are always balanced polyhedral complexes.
This leads to the following realisability question.

\begin{question} \label{que:Realisability}
Which pure-dimensional balanced polyhedral complexes are the variety 
of some tropical ideal?
\end{question}

If the tropical ideal is the tropicalisation of a prime classical
ideal, then the tropical variety is pure-dimensional and balanced
\cite{Maclagan15}*{Theorem 3.3.5}.
The question of which balanced polyhedral
complexes are realised by classical ideals has received much attention,
especially in the case of curves (see, e.g., \cites{Speyer14, Brugalle15,
Birkmeyer17}). But for general tropical ideals, very little is known
about Question~\ref{que:Realisability}: for instance, no natural algebraic
criterion that ensures that the variety is pure-dimensional is known.
%nor has their top-dimensional part been proved to be balanced. 
In fact,
until recently we had no intuition as to whether tropical ideals are flexible enough that they can realise basically any balanced polyhedral complex, or rather
more rigid, like algebraic varieties. In view of the
following theorem, we now lean towards the latter intuition.

\begin{reptheorem}{thm:main}
Let $M$ and $N$ be loopless matroids of ranks $a$ and $b$ that do not
have a quasi-product of rank $a \cdot b$. Then there exists no tropical
ideal whose tropical variety is the Bergman fan of the direct sum of $M$
and $N$, with all maximal cones having weight $1$.

In particular, there exists no tropical ideal whose tropical variety is the Bergman fan of the direct sum of the V\'amos matroid
$V_8$ and the uniform matroid $U_{2,3}$ of rank two on three elements, with all maximal cones having weight $1$.
\end{reptheorem}

In this theorem, a {\em quasi-product} of two loopless matroids 
is a matroid analogue of tensor products; see
Section~\ref{sec:QuasiProducts}. The fact that the V\'amos matroid
$V_8$ and the uniform matroid $U_{2,3}$ have no quasi-product of rank
$8$ was proved by Las Vergnas in \cite{Vergnas81}.

We believe that this theorem marks the beginning of an interesting
research programme, which, in addition to the pureness and balancing questions
mentioned above, asks which tropical ideals define matroids on the set of variables, and which matroids are, in this sense, tropically algebraic---See Problem \ref{prob:matroidal} and Question \ref{que:algebraic}.

\medskip
\noindent {\bf Acknowledgements.} Both authors would like to thank the Mittag-Leffler Institute for their hospitality during the Spring 2018 program, when this paper was conceived. In addition, FR would like to thank the Discrete Mathematics/Geometry Group at TU Berlin for their support while this paper was written.

\section{Definitions and basic results on tropical ideals} \label{sec:Basics}

Consider the tropical semifield $(\Rbar:=\RR \cup \{\infty\},\tplus,
\ttimes)$ with $\tplus:=\min$ and $\ttimes:=+$. Let $R$ be a sub-semifield
of $\Rbar$. The example most relevant to us is the Boolean semifield $\BB:=\{0,\infty\}$,
which is not only a sub-semifield but also a quotient of $\Rbar$.

\begin{definition}
Let $N$ be a finite set. 
A set $L \subseteq R^N$ is a {\bf tropical linear space} if it is an
$R$-submodule (i.e., $(\infty,\ldots,\infty) \in L$ and $f,g \in L, c
\in R \Rightarrow (c \ttimes f) \oplus g \in L$) and if, moreover, $L$
satisfies the following elimination axiom: for $i \in N$ and $f,g \in L$
with $f_i=g_i \neq \infty$, there exists an $h \in L$ with $h_i=\infty$
and $h_j \geq f_j \tplus g_j$ for all $j \in N$, with equality whenever $f_j
\neq g_j$. The $\Rbar$-submodule $L_{\Rbar}$ of $\Rbar^N$ generated by $L$
is a tropical linear space in $\Rbar^N$, and has the structure of a finite polyhedral complex;
we denote its dimension as such by $\dim L$.
\end{definition}

If $K$ is a field equipped with a 
non-Archimedean valuation onto $R$ and if $V \subseteq K^N$ is a linear
subspace, then the image of $V$ under the coordinate-wise valuation is
a tropical linear space in $R^N$, but not
all tropical linear spaces arise in this manner. Tropical linear spaces
are well-studied objects in tropical geometry and matroid theory: the
definition above is equivalent to that of \cite{Speyer04}, except that
we allow some coordinates to be $\infty$. A tropical linear
space $L$ gives rise to a matroid $M(L)$ in which the independent sets
are those subsets $A \subseteq N$ for which 
$L \cap (R^A \times
\{\infty\}^{N\setminus A})=\{\infty\}^N$, and $L$ is the set of {\bf vectors}
($R$-linear combinations of valuated circuits) of a {\bf valuated matroid}
on $M(L)$ \cite{Murota01}. With this setup, $\dim L=|N|-\rk (M(L))$. We will freely alternate
between these different characterisations of tropical linear spaces.

Set $\NN:=\{0,1,2,\ldots\}$, and let $n \in \NN$. Denote by $R[x_1,
\dots,x_n]$ the semiring of polynomials in the variables $x_1,\dots,x_n$
with coefficients in $R$. We
write $\mond$ and $\monld$ for the set of monomials in $x_1,\ldots,x_n$
of degree equal to $d$ and at most $d$, respectively, and we identify a polynomial
in $R[x_1,\ldots,x_n]$ of degree at most $d$ with its coefficient
vector in $R^{\monld}$.

\begin{definition}\label{d:tropicalidealprojective}
A subset $I \subseteq R[x_1,\dots,x_n]$ is a {\bf tropical ideal}
if $x_i \ttimes I \subseteq I$ for all $i=1,\ldots,n$ and if for each $d
\in \NN$ the set $I_{\leq d} \coloneqq \{ f \in I : \deg(f) \leq d\}$ is
a tropical linear space in $R^{\monld}$. 
\end{definition}

This definition is equivalent to \cite{MR}*{Definition 1.1}. Indeed,
there, in addition to the requirement that $I_{\leq d}$ be a tropical
linear space, it is required that $I$ is an ideal in the semiring
$R[x_1,\ldots,x_n]$. This is equivalent to the statement that
$I$ is closed under tropical
multiplication by each $x_i$ and closed under tropical
addition. However, as tropical linear spaces are already closed under tropical
addition, this does not need to be included as an explicit axiom. 

If $I$ is homogeneous, then the latter condition is equivalent to the
condition that for each $d$ the set $I_d$ of homogeneous polynomials in
$I$ of degree $d$ is a tropical linear space in $R^{\mond}$. There
is a natural notion of tropical ideals in the Laurent polynomial ring
$R[x_1^{\pm 1},\ldots,x_n^{\pm 1}]$ that we will also use, and
if $I$ is a tropical ideal in $R[x_1,\ldots,x_n]$ then the set
$I' := \{f/\bfx^{\bfu} \mid f \in I, \bfu \in \NN^n\}$ is a tropical ideal
in $R[x_1^{\pm 1},\ldots,x_n^{\pm 1}]$. 

Tropical ideals were introduced by Maclagan and Rinc\'on in
\cite{MR} as a framework for developing algebraic foundations for tropical
geometry. Tropical ideals are much better behaved than general ideals
of the polynomial semiring $R[x_1,\dots,x_n]$, as we explain below.

\begin{definition}
For $\bfw \in \RR^n$ and $f=\bigoplus_{\bfu}
c_\bfu \ttimes \bfx^\bfu \in R[x_1,\ldots,x_n]$, define the {\bf initial part}
of $f$ relative to $\bfw$ as
\[ \inn_\bfw(f):=\bigoplus_{\bfu \,:\, c_\bfu+\bfu \cdot \bfw=f(\bfw)
} \bfx^{\bfu} \quad \in \BB[x_1,\ldots,x_n]. \]
For a tropical ideal $I$ define its {\bf initial ideal} relative to $\bfw$ as
\[\inn_\bfw I:=\langle \inn_\bfw f \mid f \in I \rangle_\BB.\] 
\end{definition}

Note that in this paper we only consider weights $\bfw$ in $\RR^n$, not in $\Rbar^n$ as in
\cite{MR}. In other words, we do geometry only inside the tropical torus. 

\begin{definition}
The {\bf Hilbert function} of a tropical ideal $I \subseteq
R[x_1,\ldots,x_n]$ is the map $H_I:\NN \to \NN$ given by $d \mapsto
\binom{n+d}{d}-\dim I_{\leq d}$.
\end{definition}

Note that, as usual in commutative algebra, the Hilbert function
measures the codimension of $I_{\leq d}$ in its ambient space
$R^{\monld}$.  A homogeneous variant of this Hilbert function applies
only to homogeneous ideals and measures the codimension of $I_d$ in
$R^{\mond}$. The Hilbert function of a not necessarily homogeneous
ideal $I$ in $R[x_1,\ldots,x_n]$ 
equals the homogeneous Hilbert function of its homogenisation
in $R[x_0,\ldots,x_n]$.

The following is a special case of \cite{MR}*{Corollary
3.6}.  

\begin{theorem} \label{thm:Inn}
For a homogeneous tropical ideal $I \subseteq R[x_1,\ldots,x_n]$ and
any $\bfw \in \RR^n$, $\inn_\bfw I \subseteq \BB[x_1,\ldots,x_n]$ is a homogeneous tropical ideal,
and $H_{\inn_\bfw I} = H_I$. 
\end{theorem}

Theorem~\ref{thm:Inn} allows one to pass to monomial initial ideals and show that
the Hilbert function $H_I(d)$ of a homogeneous tropical ideal $I$
becomes a polynomial in $d$ for sufficiently large $d$, and also that homogeneous tropical
ideals satisfy the ascending chain condition. Via homogenisation, one
sees that both statements also hold for non-homogeneous tropical ideals
(but, as in the classical setting, the theorem does not apply directly,
since for instance, when $n=1$, $\inn_{(1)}(0 \oplus x_1)=0$ generates an
ideal---the entire semiring---with a smaller Hilbert function than any
tropical ideal containing $0 \oplus x_1$ but not $0$).

Furthermore, Maclagan and Rinc\'on prove that tropical ideals have
tropical varieties that are finite polyhedral complexes \cite{MR}*{Theorem 5.11}.

\begin{theorem}
If $I \subseteq R[x_1,\dots,x_n]$ is a tropical ideal then its
(tropical) {\bf variety}
\[V(I) \coloneqq \{ \bfw \in \RR^n : \inn_\bfw I 
\text{ contains no monomial} \}\]
is the support of a finite polyhedral complex.
\end{theorem}

Indeed, if $I$ is homogeneous, they show that the sets of $\bfw$
where $\inn_\bfw I$ is constant form the relatively open polyhedra of
a polyhedral complex with support $\RR^n$ called the {\bf Gr\"obner
complex} of $I$, and that the cells where $\inn_\bfw I$ contains
no monomial form a subcomplex with support $V(I)$. By homogeneity,
all cells then contain in their lineality space the linear span of the all-ones vector $\one$. 
In the case where $I \subseteq R[x_1,\ldots,x_n]$ is not necessarily homogeneous, let $I^h$ be its
homogenisation in $R[x_0,x_1,\ldots,x_n]$. Then $\bfw \mapsto (0,\bfw)$
is a bijection between $V(I)$ and the intersection of $V(I^h)$ with the
zeroeth coordinate hyperplane, and we give $V(I)$ the corresponding
polyhedral complex structure. 

The variety of a tropical ideal comes equipped with positive integral
weights on its maximal polyhedra; this is inspired by \cite{Maclagan15}*{Lemma
3.4.7} and studied more in depth in \cite{MR2}.

\begin{definition}\label{def:multiplicity}
Let $I \subseteq R[x_1,\dots,x_n]$ be a tropical ideal, let
$\sigma$ be a maximal polyhedron of $V(I)$, and
let $\bfw$ be in the relative interior of $\sigma$. The {\bf
multiplicity} of $\sigma$ in $V(I)$ is defined as follows.  First, let $I' \subseteq
R[x_1^{\pm 1},\ldots,x_n^{\pm 1}]$ be the (tropical) ideal in the
Laurent polynomial ring generated by $I$. After an automorphism
of the Laurent polynomial ring given by $\bfx^\bfu \mapsto \bfx^{A\bfu}$
with $A \in \GL_n(\ZZ)$, we can assume that the affine span of $\sigma$
is a translate of $\spann(\bfe_1,\dots,\bfe_d)$ for some $d$. In this
case, by \cite{MR2}*{Lemma 6.2}, the tropical ideal $J \coloneqq \inn_{\bfw}(I')
\cap \BB[x_{d+1}, \dots, x_{n}]$ is zero-dimensional, i.e., $H_J(e)$
is a constant for $e \gg 0$. The multiplicity of $\sigma$ is defined to be 
equal to this constant, called the degree of $J$.
\end{definition}

\begin{remark} \label{re:Mult}
A more coordinate-free version of Definition \ref{def:multiplicity}
is the following. Consider the linear span of $\sigma$, defined as 
\[ \spann(\sigma) := \RR_{\geq 0} \{ \bfv - \bfv' \mid \bfv, \bfv' \in \sigma \}. \] 
Let $S \subseteq \BB[x_1^{\pm 1},\ldots,x_n^{\pm 1}]$ be
the sub-semiring spanned by monomials $\bfx^\bfu$ of $\bfw$-weight $\bfw \cdot \bfu$ 
equal to zero for all $\bfw \in \spann(\sigma)$. 
Then $S$ itself is isomorphic to a Laurent polynomial semiring in $n-d$ variables. The
multiplicity of $\sigma$ is the degree of the zero-dimensional tropical ideal
$\inn_\bfw (I') \cap S$.
\end{remark}

We will need the following results.

\begin{lemma} \label{lm:Saturation}
Let $I$ be a tropical ideal in $R[x_1,\ldots,x_n]$. Denote by $I'$
the ideal generated by $I$ in $R[x_1^{\pm 1},\ldots,x_n^{\pm n}]$,
and set $I^\sat:=I' \cap R[x_1,\ldots,x_n]$. Then $I^\sat \supseteq I$
is a tropical ideal, and $V(I^\sat)=
V(I)$ as weighted polyhedral complexes.
\end{lemma}

We call $I^\sat$ the {\bf saturation} of $I$ with respect to $m:=x_1
\cdots x_n$, and we call $I$ {\bf saturated} with respect to $m$
if $I^\sat=I$.

\begin{proof}
That $I^\sat$ is a tropical ideal containing $I$ is straightforward from
the definition. Since $I^\sat \supseteq I$ we have $V(I^\sat) \subseteq
V(I)$.  Conversely, let $\bfw \in V(I)$ and $f \in I^\sat$. Then $\bfx^\bfu
\ttimes f \in I$ for some $\bfu \in \NN^n$, hence $\inn_\bfw (\bfx^\bfu \ttimes f)$
is not a monomial, and therefore neither is $\inn_\bfw f$.  This shows that $V(I) =
V(I^\sat)$. That the multiplicities are the same follows from the fact
that the multiplicities in $V(I)$ are defined using $I'$.
\end{proof}

If $\Sigma$ is a polyhedral complex in $\RR^n$ and $\sigma$ is a polyhedron in $\Sigma$,
the {\bf star} $\starr_\sigma \Sigma$ of $\Sigma$ at $\sigma$ is a weighted polyhedral fan, whose cones are
indexed by the cones $\tau$ of $\Sigma$ containing $\sigma$. The cone indexed by such $\tau$ is
\[\overline \tau := \RR_{\geq 0} \{ \bfv - \bfw \mid \bfv \in \tau \text{ and } \bfw \in \sigma \},\]
with weight equal to the weight of $\tau$ in $\Sigma$. 

The following can be found in \cite{MR2}*{Corollary 2.11 and Proposition 6.4}. 

\begin{proposition} \label{prop:Initial}
Let $I$ be a tropical ideal in $R[x_1,\ldots,x_n]$, $\sigma$ be a
polyhedron in $V(I)$, and $\bfw$ be in the relative
interior of $\sigma$. Then $\inn_\bfw I \subseteq \BB[x_1,\ldots,x_n]$
is homogeneous with respect to every vector $\bfv \in \spann(\sigma)$, and 
$V(\inn_\bfw I)=\starr_\bfw V(I)$
as weighted polyhedral complexes.
\end{proposition}

\section{The independence complex of a tropical ideal}

\begin{definition}
Let $I \subseteq R[x_1,\dots,x_n]$ be a tropical ideal. The {\bf
independence complex} of $I$ is the simplicial complex
\begin{equation}\label{eq:indepalgebraic}
\I(I) \coloneqq \{ A \subseteq \{1,\dots,n\} : I \cap R[x_i : i \in  A] = \{\infty \} \}.
\end{equation}
When $\I(I)$ is the collection of independent sets of a matroid $M$,
we will say that $I$ is a {\bf matroidal tropical ideal}, and that $M$
is its associated {\bf algebraic matroid}.
\end{definition}

The independence complex of a tropical ideal $I$ can be recovered from
its variety $V(I)$, at least if $R=\Rbar$.

\begin{proposition} \label{pro:Independence}
	If $I \subseteq \Rbar[x_1,\dots,x_n]$ is a tropical ideal then
	\begin{equation}\label{eq:indepgeometric}
	\I(I) = \{ A \subseteq \{1,\dots,n\} : \pi_A(V(I)) = \RR^A \},
	\end{equation}
	where $\pi_A \colon \RR^n \to \RR^A$ is the coordinate projection onto the coordinates indexed by $A$. In particular, the independence complex $\I(I)$ depends only on the variety $V(I)$.
\end{proposition}
\begin{proof}
	Let $A \subseteq \{1,\dots,n\}$. If $A \notin \I(I)$ then there exists $f \in I
	\cap \Rbar[x_i : i \in A]$
	such that $f \neq \infty$, and $V(I) \subseteq V(f)$. We then have 
	$\pi_A(V(I)) \subseteq \pi_A(V(f)) \subsetneq \RR^A$, as claimed.
	For the reverse inclusion, suppose that $\pi_A(V(I)) \subsetneq \RR^A$, and let 
	$\bfw \in \RR^A \setminus \pi_A(V(I))$. For any polynomial $f \in
	\Rbar[x_1,\dotsc,x_n]$, denote
	by $f|_\bfw$ the polynomial in $\Rbar[x_i : i \notin A]$ obtained by specializing each variable
	$x_i$ with $i \in A$ to $w_i \in \RR$. Consider the ideal $I|_\bfw \subseteq
	\Rbar[x_i : i \notin A]$
	defined as $I|_\bfw \coloneqq \{ f|_\bfw : f \in I \}$. By \cite{MR2}*{Theorem 3.6}, the ideal
	$I|_\bfw$ is a tropical ideal. Moreover, we must have $V(I|_\bfw) = \emptyset$, as any point $\bfv \in V(I|_\bfw)$
	would lift to the point $(\bfv, \bfw) \in V(I)$, contradicting that $\bfw \notin \pi_A(V(I))$.
	By the weak Nullstellensatz \cite{MR}*{Corollary 5.17}, the tropical ideal $I|_\bfw$ must contain
	the constant polynomial $0$. But then $0 = f|_\bfw$ for some $f \in I$, which in particular implies
	that $f \in I \cap \Rbar[x_i : i \in A]$ and $f \neq \infty$. 
\end{proof}
Proposition \ref{pro:Independence} also follows from the fact that a coordinate projection of the
variety of a tropical ideal is the variety of the corresponding elimination ideal \cite{MR2}*{Theorem 4.7}.

Recall that the Hilbert function $H_I(e)$ of a tropical ideal $I \subseteq \Rbar[x_1,\dots,x_n]$ eventually agrees with a polynomial in $e$, called the {\bf Hilbert polynomial} of $I$ \cite{MR}*{Proposition 3.8}. The {\bf dimension} $\dim(I)$ of $I$ is defined as the degree of its Hilbert polynomial.

\begin{corollary}
	For any tropical ideal $I$ we have
	\[\dim \I(I) + 1 = \dim V(I) = \dim I.\]
\end{corollary}
\begin{proof}
	From \eqref{eq:indepgeometric} it is clear that $\dim V(I) \geq \dim\I(I)+1$. Now, if $V(I)$ contains a polyhedron $\sigma$ of dimension $d$ then there is some coordinate projection $\pi_A(\sigma)$ that is $d$-dimensional, and thus from \eqref{eq:indepalgebraic} we see that $A \in \I(I)$ and thus $\dim \I(I) + 1 \geq d$. This shows that $\dim \I(I) + 1 = \dim V(I)$. The equality $\dim V(I) = \dim I$ is proved in \cite{MR2}*{Theorem 4.3}.
\end{proof}

In the classical setting, primality of an ideal implies matroidality. 
We do not know about a similarly appealing sufficient condition
for matroidality of general tropical ideals.

\begin{example}\label{ex:matroidal}
	If $J \subseteq K[x_1,\dotsc,x_n]$ is a prime ideal, where $K$ is a
	field with a non-Archimedean valuation, then $\trop(J)$ is a matroidal tropical ideal.
	Its associated algebraic matroid is the matroid that captures algebraic independence
	among the coordinate functions $x_1, \dots, x_n$ in the field of fractions of $K[x_1,\dots,x_n]/J$. 
\end{example}

\begin{problem}\label{prob:matroidal}
Find algebraic conditions on a tropical ideal that imply matroidality.
\end{problem}

As shown in Example \ref{ex:matroidal}, any (classically) algebraic matroid is the algebraic matroid of a tropical ideal.
However, in principle, it is possible that the class of matroids that are ``tropically algebraic'' is strictly larger than the usual class of algebraic matroids. 

\begin{question}\label{que:algebraic}
Which matroids arise as the algebraic matroid of a tropical ideal?
\end{question}

\section{Quasi-products of matroids}  \label{sec:QuasiProducts}

To motivate the definition of quasi-products, let $v_1,\ldots,v_m$
be nonzero vectors in a vector space $V$ and let $w_1,\ldots,w_n$ be
nonzero vectors in a vector space $W$ over the same field. The $v_i$
define a matroid $M$ with ground set $[m]$ in which $S \subseteq
[m]$ is dependent if and only if the set $\{v_i : i \in S\}$ is linearly
dependent. Similarly, the $w_j$ define a matroid $N$ with ground
set $[n]$. Now consider the vectors $v_i \otimes v_j \in V \otimes W,\
i \in [m], j \in [n]$. In the same manner, these define a matroid $P$
with ground set $[m] \times [n]$. One can check that $P$ is in general
not determined by $M$ and $N$, i.e., the linear
dependencies among the $v_i \otimes w_j$ cannot be read off from 
those among the $v_i$ and those among the
$w_j$. However, some features of $P$
{\em are} predicted by $M$ and $N$: for each fixed $i \in [m]$, the linear dependencies among the
vectors $v_i \otimes w_j,\ j \in [n]$ are precisely those recorded by
$N$;  here we use that $v_i$ is nonzero. Similarly, for each $j \in
[n]$, the restriction of $P$ to $[m] \times \{j\}$ is isomorphic to
$M$. Furthermore, if $B$ is a basis of $M$ and $C$ is a basis of $N$,
then $B \times C$ is a basis of $P$. In particular, the rank of $P$ is the
product of the ranks of $M$ and $N$.  Following Las Vergnas, we use these
observations to define quasi-products of general matroids, as follows.

\begin{definition}[\cite{Vergnas81}] \label{de:QuasiProduct}
Let $M,N$ be loopless matroids with ground sets $[m],[n]$, respectively. A
{\em quasi-product} of $M$ and $N$ is a matroid $P$ with ground set $[m]
\times [n]$ with the property that for each $i \in [m]$ the map $[n]
\to [m] \times [n], j \mapsto (i,j)$ is an isomorphism from $M$ to the
restriction of $P$ to $\{i\} \times [n]$ and for each $j \in [n]$ the
map $[m] \to [m] \times [n], i \mapsto (i,j)$ is an isomorphism from $M$
to the restriction of $P$ to $[m] \times \{j\}$.
\end{definition}

The properties of a quasi-product $P$ of $M$ and $N$ imply that if $B
\subseteq [m]$ is a basis of $M$ and $C \subseteq [n]$ is a basis of
$N$, then $B \times C$ is a spanning set of $P$, so the
rank of $P$ is at most the product of the ranks of $M$ and
$N$. By the discussion above, two matroids that are representable over the
same field always admit a quasi-product whose rank is the product of
their ranks. In general, however, a quasi-product with this property
need not exist.

\begin{theorem}[\cite{Vergnas81}] \label{thm:Vergnas}
Any quasi-product of the rank-4 V\'amos matroid $V_8$ and the
rank-2 uniform matroid $U_{2,3}$ has rank at most $7< 4
\cdot 2$.
\end{theorem}

\section{Not every Bergman fan is the variety of a tropical ideal}

We now prove that not every balanced polyhedral complex can be
obtained as the variety of a tropical ideal. Our counterexample will 
be the Bergman fan of a matroid; see \cite{Ardila06} for details.

\begin{definition}
Let $M$ be a loopless matroid of rank $d$ on the ground set $\{1,\dots,n\}$. 
The {\bf Bergman fan} $\B(M)$ of $M$ is the pure $d$-dimensional polyhedral fan in $\RR^n$ 
consisting of the cones of the form 
\[\sigma_\F \coloneqq  \cone(\bfe_{F_1},\bfe_{F_2},\dotsc,\bfe_{F_k}) + \RR \!\cdot\! 
\bfe_{\{1,\dotsc, n\}}\] where 
$\F = \{\emptyset \subsetneq F_1 \subsetneq F_2 \subsetneq \dotsb \subsetneq F_k 
\subsetneq \{1,\dotsc, n\}\}$ is a chain of flats in 
the lattice of flats $\mathcal L(M)$ of $M$, and where $\bfe_S$ stands
for the sum of the standard basis vectors $\bfe_i$ with $i$ running
through $S$. 
The Bergman fan of any matroid is given the structure of a balanced polyhedral complex 
by defining the multiplicity of each maximal cone to be equal to 1.
\end{definition}

Bergman fans of matroids are the tropical linear spaces (more specifically, their part inside the torus $\RR^n$) that correspond to valuated matroids where the basis valuations all take values in $\BB$.

%Let $U_{2,3}$ be the uniform matroid of rank 2 on the ground set
%$\{1,2,3\}$, and let $V_8$ be the V\'amos matroid (of rank 4 on 8
%elements). 
The following is our main result.

\begin{theorem}\label{thm:main}
Let $M$ be a loopless matroid of rank $a$ with ground set $[m]$ and let $N$
be a loopless matroid of rank $b$ with ground set $[n]$.  Suppose that
every quasi-product of $M$ and $N$ has rank strictly less than
$a\cdot b$. Then there exists no tropical ideal $I \subseteq
\Rbar[x_1,\dots,x_m,y_1,\dots,y_n]$ such that $V(I)$ is equal to $\B(M
\oplus N)$ as weighted polyhedral complexes, even up to common refinement.

In particular, there is no tropical ideal $I \subseteq \Rbar[x_1, \dots,
x_3, y_1, \dots, y_8]$ such that $V(I)$ is equal to $\B(U_{2,3} \oplus
V_8)$ as weighted polyhedral complexes, even up to common refinement.
\end{theorem}

Note that we do not require the polyhedral structure on $V(I)$
coming from the Gr\"obner complex of the homogenisation of $I$ to be equal
to the fan structure on the Bergman fan described above.

To prove the theorem, in addition to the fundamental results from
Section~\ref{sec:Basics}, we will need results relating $V(I)$
to $H_I$ for any tropical ideal $I$.

\begin{lemma} \label{lm:Intersection}
Let $L,L' \subseteq R^N$ be tropical linear spaces. If $\dim L +
\dim L'>|N|$, then $L \cap L' \neq \{(\infty,\ldots,\infty)\}$. 
\end{lemma}
\begin{proof}
The notion of stable intersection for tropical linear spaces was 
studied by Speyer in \cite{Speyer04} when the underlying matroids of both
tropical linear spaces were uniform matroids, and later generalized by
Mundinger \cite{Mundinger} for arbitrary tropical linear spaces in $R^N$.
The stable intersection $L \cap_{\st} L'$ is a tropical linear space 
contained in both $L$ and $L'$, and it has dimension a least
$\dim L + \dim L' - |N| > 0$, which implies the desired result.
\end{proof}

\begin{proposition}\label{pro:lowerbound}
Let $I \subseteq R[x_1,\dots,x_n]$ be a tropical ideal. 
If the independence complex $\I(I)$ contains a subset $A$ of size $r$, then
$H_I(d) \geq \binom{r+d}{d}$ for all $d \in \NN$.
\end{proposition}

\begin{proof}
The space $R[x_i:i \in A]_{\leq d}$ is a tropical linear space
in $R^{\monld}$ of dimension $\binom{r+d}{d}$ and, by assumption,
it does not intersect $I_{\leq d}$.  Hence by Lemma~\ref{lm:Intersection},
$\dim I_{\leq d} \leq \binom{n+d}{d} - \binom{r+d}{d}$, and therefore
$H_I(d) \geq \binom{r+d}{d}$. 
\end{proof}

\begin{proposition}\label{pro:upperbound}
Let $I \subsetneq R[x_1,\dots,x_n]$ be a tropical ideal, and set
$r:=H_I(1)-1$. Then $H_I(d) \leq \binom{r+d}{d}$ for all $d \in \NN$. 
\end{proposition}

\begin{proof}
Let $I^h \subseteq R[x_0,\ldots,x_n]$ be the homogenisation
of $I$. Then $\dim (I^h)_d=\dim I_{\leq d}$ for all $d \in \NN$, and in
particular $\dim (I^h)_1=\dim I_{\leq 1}=n+1-H_I(1)=n-r$. Moreover,
by applying Theorem~\ref{thm:Inn} with a sufficiently general weight vector $\bfw$, 
the Hilbert function of $I^h$ is also that of
some monomial ideal $J$. We find that $J$ contains precisely $n-r$ of the
$n+1$ variables $x_0,\ldots,x_n$, and therefore all their multiples. 
This implies that $\dim
J_d \geq \binom{n+d}{d}-\binom{r+d}{d}$, where the last term counts
monomials in the remaining $r+1$ variables of degree $d$. 
We then have
\[ \textstyle H_I(d)=\binom{n+d}{d}-\dim I_{\leq d}=\binom{n+d}{d}-\dim J_d
\leq \binom{n+d}{d}-\binom{n+d}{d}+\binom{r+d}{d}, \]
as desired.
\end{proof}

The following proposition shows that the algebraic matroid of a Bergman fan $\B(M)$ (as in Proposition \ref{pro:Independence}) is equal to the matroid $M$.

\begin{proposition}[{\cite{Yu}*{Lemma 3}}]\label{pro:complex}
The independence complex of the Bergman fan $\B(M)$ of a loopless
matroid $M$ is the same as the independence complex of $M$.
\end{proposition}

We now present a key step towards proving our main result.

\begin{proposition}\label{pro:degree1}
Let $M$ be a loopless matroid on the ground set $\{1,\dots,n\}$. 
Suppose $J \subseteq \BB[x_1,\dots,x_n]$ is a homogeneous tropical ideal,
saturated with respect to $x_1 \cdots x_n$, whose variety $V(J)$ has a
common refinement, as weighted polyhedral complexes, with the Bergman
fan $\B(M)$ (with weight 1 in all its
maximal cones).  Then the matroid $M(J_1)$ is equal to $M$, under the
identification $x_i \leftrightarrow i$ of ground sets.
\end{proposition}

\begin{proof}
Let $B = \{b_1, \dots, b_d\}$ be a basis of $M$. For $0 \leq i \leq
d$, consider the flat $F_i$ of $M$ obtained as the closure of the set
$\{b_1, \dots, b_i\}$, and let $\sigma$ be the maximal cone of $\B(M)$
corresponding to the chain of flats $\emptyset = F_0 \subsetneq F_1
\subsetneq \dots \subsetneq F_{d-1} \subsetneq F_d = \{1,\dots,n\}$.
Let $\tau \subseteq \sigma$ be a maximal cone in a common refinement of both $V(J)$ and $\B(M)$.
The linear span $\spann(\tau) = \spann(\sigma)$ consists of all vectors 
$\bfw \in \RR^n$ for which $w_i = w_j$ whenever $\{i, j\} \subseteq F_k\setminus F_{k-1}$ for some $k=1,\ldots,d$.
A monomial $\bfx^\bfu$ in $\BB[x_1^{\pm 1},\dots,x_n^{\pm 1}]$ has
$\bfw$-weight equal to zero for all such $\bfw$ 
if and only if for every $k$ we have
$\sum_{i \in F_k \setminus F_{k-1}} u_i = 0$. As in
Remark~\ref{re:Mult}, let $S$ be the subsemiring of $\BB[x_1^{\pm 1},\dots,x_n^{\pm 1}]$
consisting of all polynomials involving only such monomials, 
and let $J'$ be the (tropical) ideal in
$\BB[x_1^{\pm 1},\ldots,x_n^{\pm 1}]$ generated by $J$.

Take $\bfv$ to be a vector in the relative interior of $\tau$.
Since $\tau$ has multiplicity $1$ in $V(J)$, 
$\inn_\bfv(J') \cap S$ is zero-dimensional of degree $1$, and contains no monomials. 
Hence for any pair of distinct monomials $\bfx^\bfu,
\bfx^{\bfu'}$ in $S$, $\inn_{\bfv} (J') \cap S$ contains the binomial
$\bfx^\bfu \tplus \bfx^{\bfu'}$. In particular, if $\{i \neq j\} \subseteq F_k \setminus F_{k-1}$
for some $k$ then $0 \oplus x_i^{-1} x_j \in \inn_\bfv(J')
\cap S$, and thus $x_i \tplus x_j \in \inn_\bfv(J')$.
As $J$ is homogeneous and saturated with respect to $x_1
\cdots x_n$, this implies that there is a polynomial of the form $x_i \tplus x_j \tplus f$
in $J_1$ where $f$ is a sum of variables all contained in
$F_{k-1}$. 
It follows that $x_i$ is in the closure of $F_{k-1} \cup \{x_j\}$ in the matroid $M(J_1)$.
We conclude that $\{b_1, \dots, b_d\}$ is a generating set in the matroid $M(J_1)$, 
and thus $\rank(M(J_1)) \leq \rank(M)$. 
Now, the tropical prevariety cut out by the linear polynomials in $J$ is equal to $\B(M(J_1))$, so we have 
$\B(M(J_1)) \supseteq V(J) = \B(M)$. It follows from \cite{Rincon}*{Lemma 7.4} that 
$\B(M(J_1)) = \B(M)$, and thus $M(J_1) = M$, completing the proof.
\end{proof}

We conclude with the proof of the main theorem.

\begin{proof}[Proof of Theorem~\ref{thm:main}]
Suppose that such an $I$ exists, and denote $O:=M \oplus N$. We first
argue that we may replace $I$ by an ideal $J$ that is homogeneous
as well as saturated. To this end, let
$\sigma$ be a polyhedron in $V(I)$ whose affine span is $\RR \cdot \one$ (which is contained in the lineality space of $\B(O)$), 
and let $\bfw$ be in the relative
interior of $\sigma$.  Set $J' := \inn_\bfw I  \subseteq \BB[x_1, \dots, x_m,
y_1, \dots, y_n]$.  By Proposition~\ref{prop:Initial}, the tropical ideal $J'$ is homogeneous (with respect to $\one$) and has variety $V(J') = \starr_\bfw V(I)$,
which is equal to $\B(O)$ up to common refinement. 
Consider the homogeneous ideal $J := (J')^\sat$. By Lemma~\ref{lm:Saturation}, we have that $V(J)$ is also equal to $\B(O)$ up to common refinement. 

Now, by Proposition~\ref{pro:degree1}, $M(J_1)$ is equal
to $O$.  Since $\rk O=a+b$, we find that $H_J(1)=1+a+b$ and thus, by
Proposition~\ref{pro:upperbound}, $H_J(d) \leq \binom{a+b+d}{d}$ for
all $d$.
On the other hand, since $V(J) = \B(O)$, by Propositions~\ref{pro:complex} 
and~\ref{pro:Independence} the tropical ideal $J$ is matroidal, with associated
algebraic matroid $O = M \oplus N$. Hence, by Proposition
\ref{pro:lowerbound} we have $H_J(d) \geq \binom{a+b+d}{d}$. We conclude that
$H_J(d)=\binom{a+b+d}{d}$. 

Denote $Q:=M(J_2)$. The matroid $Q$ has rank 
$H_J(2)-H_J(1) = \binom{a+b+1}{2}$ on the ground set $S_1 \sqcup S_2 \sqcup S_3$, where $S_1:=\{x_i x_j \mid
1 \leq i \leq j \leq m \}$, $S_2:=\{y_i y_j \mid 1 \leq i \leq j \leq
n\}$, and $S_3:=\{x_i y_j \mid 1\leq i \leq m, 1\leq j \leq n\}$. The restriction
$Q|{S_1}$ is spanned by all products of
two elements in a basis of $M(J_1)|{\{x_1,x_2,\ldots,x_m\}}$, hence has rank
at most $\binom{a+1}{2}$.  Similarly, the restriction $Q|{S_2}$ has
rank at most $\binom{b+1}{2}$. Hence $Q|{S_3}$ has rank at least
$\binom{a+b+1}{2}-\binom{a+1}{2}-\binom{b+1}{2}=ab$. 

Since $J$ is saturated, for each $1\leq i \leq m$, multiplication by $x_i$
yields an isomorphism between the matroid $M(J_1)|{\{y_1,\ldots,y_n\}}
\cong N$ and the restriction of $Q$ to $x_i \cdot \{y_1,\ldots,y_n\}
\subseteq S_3$.  Similarly, for each $1\leq j \leq n$, the
restriction of $Q$ to $y_j \cdot \{x_1,\dots,x_m\}$ is isomorphic to
$M$. Hence $Q|{S_3}$ is a quasi-product of $M$ and $N$ in the sense of
Definition~\ref{de:QuasiProduct}. But the assumption in the theorem
is that such a quasi-product has rank strictly less than $a \cdot b$,
a contradiction. Hence no such ideal $I$ exists.

The second part of the main theorem is a direct consequence of the first
part and Theorem~\ref{thm:Vergnas} by Las Vergnas.
\end{proof}

\section{Concluding remarks}

Using the result by Las Vergnas that $U_{2,3}$ and $V_8$ do not have a
quasi-product of rank $8$, we have showed that the Bergman fan of their
direct sum is not the tropical variety of any tropical ideal, with weight
$1$ on all the maximal cones.

We do not know whether there exists a tropical ideal whose tropical
variety is the Bergman fan of $U_{2,3} \oplus V_8$ as a set, without the
condition that all weights be $1$. 
%To rule this out, it seems that one would have to
%understand higher-degree analogues of quasi-products much better than
%we currently do. 

We also do not know whether $\B(V_8)$ itself is the tropical variety of
any tropical ideal with weight one on the maximal cones. To study this
question for a matroid $M$, one needs to develop the theory of {\em
symmetric squares} of matroids, in a fashion similar to Las Vergnas's
quasi-products from Section~\ref{sec:QuasiProducts}. But already for $V_8$ this seems considerably harder than
quasi-products of $U_{2,3}$ with $V_8$.

Finally, we'd like to point out that for any $m \geq 3$, the matroids
$U_{2,m}$ and $V_8$ do not admit a quasi-product of rank $8$. Indeed,
if $P$ were such a quasi-product on $[m] \times [8]$, then for any
basis $C \subseteq [8]$ of $V_8$ the set $[2] \times C$, which spans
$P$, would have to be a basis. But then the restriction of $P$ to
$[3] \times [8]$ would be a quasi-product of $U_{2,3}$ and $V_8$ of
rank $8$, a contradiction to Las Vergnas's Theorem~\ref{thm:Vergnas}.
This simple observation yields infinitely many matroids to which our
Theorem~\ref{thm:main} applies. However, it would be interesting to
find more intricate families of pairs of matroids that do not admit
quasi-products of the correct rank.

\end{document}